\documentclass[11pt,a4paper]{article}

\usepackage{extarrows}
\usepackage{enumerate}
\usepackage{amsmath}
\usepackage{amssymb}
\usepackage{amsthm}
\usepackage[all]{xypic}
\usepackage[all,cmtip,poly]{xy}
\usepackage{latexsym}
\usepackage{pstricks,hyperref}
\usepackage{amsbsy}
\usepackage{amscd}
\usepackage{mathrsfs}
\usepackage{txfonts}
\usepackage{amsfonts}
\usepackage{graphicx}
\usepackage[top=1.45in,bottom=1.55in,left=1.25in,right=1.25in]{geometry}
\newtheorem{thm}{Theorem}[section]
\newtheorem{cor}[thm]{Corollary}
\newtheorem{lem}[thm]{Lemma}
\newtheorem{exm}[thm]{Example}
\newtheorem{prop}[thm]{Proposition}
\newtheorem{defn}[thm]{Definition}
\newtheorem{rem}[thm]{Remark}
\newtheorem{que}[thm]{Question}

\usepackage{pstricks,color}

\begin{document}

\begin{center}
{\Large \bf Relative cluster tilting objects in triangulated categories\footnote{Supported by the NSF of China (Grants 11671221) }}

\bigskip

{\large Wuzhong Yang and Bin Zhu}
\bigskip

\end{center}

\def\s{\stackrel}
\def\Longrightarrow{{\longrightarrow}}
\def\A{\mathcal{A}}
\def\B{\mathcal{B}}
\def\C{\mathcal{C}}
\def\D{\mathcal{D}}
\def\T{\mathcal{T}}
\def\R{\mathcal{R}}
\def\S{\mathcal{S}}
\def\H{\mathcal{H}}
\def\U{\mathscr{U}}
\def\V{\mathscr{V}}
\def\W{\mathscr{W}}
\def\X{\mathscr{X}}
\def\Y{\mathscr{Y}}
\def\Z{\mathcal {Z}}
\def\I{\mathcal {I}}
\def\Aut{\mbox{Aut}}
\def\coker{\mbox{coker}}
\def\deg{\mbox{deg}}
\def\dim{\mbox{dim}}
\def\Ext{\mbox{Ext}}
\def\Hom{\mbox{Hom}}
\def\Gr{\mbox{Gr}}
\def\id{\mbox{id}}
\def\Im{\mbox{Im}}
\def\ind{\mbox{ind}}
\def\Int{\mbox{Int}}
\def\ggz{\Gamma}
\def\la{\Lambda}
\def\bz{\beta}
\def\az{\alpha}
\def\gz{\gamma}
\def\da{\delta}
\def\fs{{\mathfrak{S}}}
\def\ff{{\mathfrak{F}}}
\def\zz{\zeta}
\def\thz{\theta}
\def\raw{\rightarrow}
\def\ole{\overline}
\def\cat{\C_{F^m}(\H)}
\def\fun{F_\la}
\def\sttm{\mbox{s}\tau\mbox{-tilt}\la}
\def\wte{\widetilde}
\def \text{\mbox}

\hyphenation{ap-pro-xi-ma-tion}

\newcommand{\add}{\mathsf{add}\hspace{.01in}}
\newcommand{\Fac}{\mathsf{Fac}\hspace{.01in}}
\newcommand{\End}{\operatorname{End}\nolimits}
\renewcommand{\mod}{\mathsf{mod}\hspace{.01in}}
\newcommand{\proj}{\mathsf{proj}\hspace{.01in}}

\begin{abstract}
Assume that $\D$ is a Krull-Schmidt, Hom-finite triangulated category with a Serre functor and a cluster-tilting object $T$. We introduce the notion of relative cluster tilting objects, and $T[1]$-cluster tilting objects in $\D$, which are a generalization of cluster-tilting objects. When $\D$ is $2$-Calabi-Yau, the relative cluster tilting objects are cluster-tilting. Let $\la={\rm End}^{op}_{\D}(T)$ be the opposite algebra of the endomorphism algebra of $T$. We show that there exists a bijection between $T[1]$-cluster tilting objects in $\D$ and support $\tau$-tilting $\la$-modules, which generalizes a result of Adachi-Iyama-Reiten \cite{AIR}. We develop a basic theory on $T[1]$-cluster tilting objects. In particular, we introduce a partial order on the set of $T[1]$-cluster tilting objects and mutation of $T[1]$-cluster tilting objects, which can be regarded as a generalization of `cluster-tilting mutation'. As an application, we give a partial answer to a question posed in \cite{AIR}.
\end{abstract}

\medskip

\section{Introduction}\label{sect:1}

Cluster-tilting objects in a triangulated category $\D$ were introduced in \cite{BMRRT, BMR, IY, KR, KZ}. When $\D$ is a cluster category or more general, a $2$-Calabi-Yau ($2$-CY for short) triangulated category, they play a crucial role in the categorification of cluster algebras, and they correspond to the clusters \cite{K2}.
\medskip

Cluster algebras were introduced by Fomin and Zelevinsky in \cite{FZ}. There has been a vast literature to establish connections with  representation theory of finite dimensional algebras. Marsh, Reineke and Zelevinsky made a first attempt to understand cluter algebras in terms of the representation theory of quivers in \cite{MRZ}. Immediately following this, Buan, Marsh, Reiten, Reineke and Todorov in \cite{BMRRT, K1} invented cluster categories (see also \cite{CCS} for type $A_n$). This led to develop a theory, namely cluster-tilting theory, and yielded a categorification of acyclic cluster algebras. At the same time, Gei{\ss}, Leclerc and Schr\"{o}er \cite{GLS1, GLS2} studied cluster-tilting objects in module categories over preprojective algebras and gave a categorification of certain cyclic cluster algebras.  Cluster categories and the stable module categories of preprojective algebras of Dynkin quivers are examples of $2$-CY triangulated categories.
\medskip

Cluster-tilting objects have many nice properties. A fruitful theory about them has been developed in last ten years, see for example, \cite{KR}, \cite{BIRS} for cluster-tilting in $2$-CY triangulated categories; \cite{KZ}, \cite{IY}, \cite{B} for cluster-tilting in triangulated categories. One of the important properties of cluster-tilting objects in $2$-CY triangulated categories is that when we remove some direct summand $T_i$ from cluster-tilting object $T=T_1\oplus T_2\oplus \cdots \oplus T_n$ to get $T/T_i=\oplus_{j\neq i}T_j$ (which is called an almost complete cluster-tilting object), then there is exactly one indecomposable object $T_i^*$ such that $T_i^*\ncong T_i$ and $T/T_i\oplus T_i^*$ is a cluster-tilting object, which is called the mutation of $T$ at $T_i$. Mutation of cluster-tilting objects in $2$-CY triangulated categories were defined in \cite{BMRRT, IY} and studied by many authors after them, it corresponds to the mutation of clusters in the categorification of cluster algebras. But the mutation of cluster-tilting objects in triangulated categories which are not $2$-CY, is not always possible, see for example Section II1 in \cite{BIRS}, and see examples in Section \ref{sect:5}. To make mutation always possible, it is desirable to enlarge the class of cluster-tilting objects to get the more general property that almost complete ones always have two complements.
\medskip

Cluster-tilted algebras were introduced by Buan, Marsh and Reiten in \cite{BMR}, which are by definition, the endomorphism algebras of cluster-tilting objects in cluster categories. It was proved that the module category of a cluster-tilted algebra is equivalent to the quotient category of cluster category by the cluster-tilting object \cite{BMR}. One can also study the endomorphism algebra of a cluster-tilting object in a triangulated category, the equivalence above still holds in this general case \cite{KR, KZ, IY}. Under this equivalence, one can ask the problem whether a tilting module over the endomorphism algebra of a cluster-tilting object can be lifted to a cluster-tilting object in the triangulated category. Smith \cite{Smi} and Fu-Liu \cite{FL} proved that it is always true for cluster categories and $2$-CY triangulated categories. Holm-J{\o}rgensen \cite{HJ} and Beligiannis \cite{B} proved it is true when the global dimension of endomorphism algebra is finite. More recently, Adachi-Iyama-Reiten \cite{AIR} introduced the $\tau$-tilting modules for any finite dimensional algebra. Assume that $\D$ is a $2$-CY triangulated category with a cluster-tilting objects $T$. In \cite{AIR}, the authors  established a bijection between cluster-tilting objects in $\D$ and support $\tau$-tilting modules in $\mod\End^{op}_{\D}(T)$ (see also \cite{CZZ, YZZ} for various version of the bijection). Unfortunately, many examples (see for example Section \ref{sect:5}, and see Example \ref{exm:2.15}) indicate that this result does not hold if $\D$ is not 2-CY. It is then reasonable to find a class of objects in $\D$ which correspond to support $\tau$-tilting modules in $\mod\End^{op}_{\D}(T)$ bijectively in more general setting.
\medskip

For these purposes, we introduce the notion of relative cluster tilting objects in a triangulated category $\D$, which are a generalization of cluster-tilting objects. For an object $M$ in $\D$, we use $[M](X, Y)$ to denote the subgroup of $\Hom_{\D}(X, Y)$ consisting of the morphisms from $X$ to $Y$ factoring through $\add M$. In this way, we define an ideal $[M[1]](-, -)$ of $\D$ if $M$ is a cluster-tilting object, which is called ghost ideal of $\D$ in \cite{B}.
\begin{defn}
Let $\D$ be a triangulated category with cluster-tilting objects.
\begin{itemize}
\item An object $X$ in $\D$ is called {\rm relative rigid} if there exists a cluster-tilting object $T$ such that $[T[1]](X, X[1])=0$. In this case, $X$ is also called $T[1]$-{\rm rigid}.    
\item An object $X$ in $\D$ is called {\rm relative cluster tilting} (respectively, {\rm almost relative cluster tilting}) if there exists a cluster-tilting object $T$ such that $X$ is $T[1]$-rigid and $|X|=|T|$ (respectively, $|X|=|T|-1$), where $|X|$ denotes the number of non-isomorphic indecomposable direct summands of $X$. In this case, $X$ is also called $T[1]$-{\rm cluster tilting} (respectively, {\rm almost} $T[1]$-{\rm cluster tilting}).
\end{itemize}
\end{defn}
For a cluster-tilting object $T$, we introduce a partial order on the set of basic $T[1]$-cluster tilting objects and get the first main result of this paper.
\begin{thm}\label{thm:intro1}(see Theorem \ref{thm:mainthm} and Proposition \ref{prop:twoposets} for details).
Let $\D$ be a triangulated category with a Serre functor and a cluster-tilting object $T$,  and let $\la={\rm End}^{op}_{\D}(T)$. Then there is an order-preserving bijection between
the set of isomorphism classes of basic $T[1]$-cluster tilting objects in $\D$ and the set of isomorphism classes of basic support $\tau$-tilting $\la$-modules.
\end{thm}
If $\D$ is a 2-CY triangulated category, then it turns out that $T[1]$-cluster tilting objects are precisely cluster-tilting objects. Thus this theorem improves a result in \cite{AIR}. Furthermore, we introduce mutation of relative cluster tilting objects. The second main result of this paper is the following, which is a generalization of a result in \cite{BMRRT, IY}.%
\begin{thm}\label{thm:intro2}(see Corollary \ref{cor:maincor1} and Theorem \ref{thm:mutation} for details).
Let $\D$ be a triangulated category with a Serre functor and a cluster-tilting object $T$. Then any basic almost $T[1]$-cluster tilting object in $\D$ has exactly two non-isomorphic indecomposable complements, and they are related by exchange triangles.
\end{thm}
In the last part of this paper, we give an application. In \cite{AIR}, the authors gave a method to calculate left mutation of support $\tau$-tilting modules by exchange sequences and raised a question about exchange sequences (Question 2.28 in \cite{AIR}). For this question, we first give a relationship between exchange sequences and the exchange triangles in Theorem \ref{thm:intro2} and then give a partial answer.

\medskip

The paper is organized as follows. In Section \ref{sect:2}, we review some elementary definitions and facts that we need to use, including cluster-tilting objects and support $\tau$-tilting modules. In Section \ref{sect:3}, we first give some basic properties of relative cluster tilting objects, then we state and prove our first main result. In Section \ref{sect:4}, we introduce mutation of relative cluster tilting objects and prove our second main result. As an application, we give a partial answer to Question 2.28 in \cite{AIR}. In the last section, we present some examples.
\medskip

We end this section with some conventions. Throughout this article, $k$ is an algebraically closed field and $D=${\rm Hom}$_k(-,k)$ is the $k-$duality. All modules we consider in this paper are left modules. For a finite dimensional algebra $\la$, $\mod\la$ denotes the category of finitely generated left $\la$-modules, and $\proj\la$ denotes the subcategory of $\mod\la$ consisting of projective $\la$-modules. For any triangulated category $\D$, we assume that it is $k$-linear, Hom-finite, and satisfies the Krull-Remak-Schmidt property \cite{H}. In $\D$, we denote the shift functor by $[1]$ and define Ext$_{\D}^i(X, Y)\coloneqq $Hom$_{\D}(X,$ $Y[i])$ for any objects $X$ and $Y$. For simplicity, we use $\D(X, Y)$ or $(X, Y)$ to denote the set of morphisms from $X$ to $Y$ in $\D$. If $\T$ is a subcategory of $\D$, then we always assume that $\T$ is a full subcategory which is closed under taking isomorphisms, direct sums and direct summands. For three objects $M, X$ and $Y$ in $\D$, we denote by $\add M$ the full subcategory of $\D$ consisting of direct summands of direct sum of finitely many copies of $M$, and denote by $[M](X, Y)$ the subgroup of $\Hom_\D(X, Y)$ consisting of morphisms which factor through objects in $\add M$. The quotient category $\D/[M]$ of $\D$ is a category with the same objects as $\D$ and the space of morphisms from $X$ to $Y$ is the quotient of group of morphisms from $X$ to $Y$ in $\D$ by the subgroup consisting of morphisms factor through objects in $\add M$. For two morphisms $f:M\rightarrow N$ and $g:N\rightarrow L$, the composition of $f$ and $g$ is denoted by $gf:M\rightarrow L$.
\section{Preliminaries}\label{sect:2}
In this section, we recall some definitions and results that will be used in this paper.

\subsection{Support $\tau$-tilting modules}
Let $\la$ be a finite dimensional $k$-algebra and $\tau$ be the Auslander-Reiten translation. Support $\tau$-tilting modules were introduced by Adachi, Iyama and Reiten \cite{AIR}, which can be regarded as a generalization of tilting modules.
\begin{defn}Let $(X, P)$ be a pair with $X\in \mod\la$ and $P\in \proj\la$.
\begin{enumerate}
\item We say that $(X, P)$ is {\rm basic} if $X$ and $P$ are basic.
\item $X$ is called $\tau${\rm -rigid} if {\rm Hom}$_{\la}(X,\tau X)=0$. $(X, P)$ is called a $\tau${\rm -rigid pair} if $X$ is $\tau$-rigid and {\rm Hom}$_{\la}(P, X)=0$. 
\item $X$ is called $\tau${\rm -tilting} if $X$ is $\tau$-rigid and $|X|=|\la|$.
\item A $\tau$-rigid pair $(X, P)$ is said to be a {\rm support $\tau$-tilting} (respectively, {\rm almost support $\tau$-tilting}) pair if $|X|+|P|=|\la|$ (respectively, $|X|+|P|=|\la|-1$). If $(X, P)$ is a support $\tau$-tilting pair, then $X$ is called a {\rm support $\tau$-tilting module}.
\item Let $(X', P')$ be a pair with $X'\in \mod\la$ and $P'\in \proj\la$. We say that $(X,P)$ is a direct summand of $(X',P')$ if $X$ is a direct summand of $X'$ and $P$ is a direct summand  of $P'$.
\end{enumerate}
\end{defn}
Throughout this paper, we denote by $\tau$-tilt$\la$ (respectively, $\sttm$) the set of isomorphism classes of basic $\tau$-tilting (respectively, support $\tau$-tilting) $\la$-modules, and by $\tau$-rigid$\la$ the set of isomorphism classes of basic $\tau$-rigid pairs of $\la$. The following observation is basic in $\tau$-tilting theory.
\begin{prop}\label{uniqueness}\cite[Proposition 2.3]{AIR}
Let $(X, P)$ be a basic pair with $X\in \mod\la$ and $P=\la e\in \proj\la$, where $e$ is an idempotent of $\la$.
\begin{enumerate}
\item[(a)] $(X, P)$ is a $\tau$-rigid (respectively, support $\tau$-tilting) pair for $\la$ if and only if $X$ is a $\tau$-rigid (respectively, $\tau$-tilting) $(\la/\la e\la)$-module.
\item[(b)] Let $(X, P)$ be a support $\tau$-tilting pair for $\la$. Then $P$ is determined by $X$ uniquely. This means that if $(X, P)$ and $(X, Q)$ are basic support $\tau$-tilting pairs for $\la$, then $P\simeq Q$.
\end{enumerate}
\end{prop}
For $\tau$-tilting modules, we have the analog of the Bongartz completion of tilting modules.
\begin{thm}\label{thm:tauBongartz}
\cite[Theorem 2.9]{AIR} Let $X$ be a $\tau$-rigid $\la$-module. Then there exists a $\la$-module $V$ such that $X\oplus V$ is a $\tau$-tilting $\la$-module.
\end{thm}

The notion of mutation was also introduced in \cite{AIR}.
\begin{defn}
Two basic support $\tau$-tilting pairs $(T, P)$ and $(T', P')$ for $\la$ are said to be {\rm mutation} of each other if there exists a basic almost support $\tau$-tilting pair $(U, Q)$ which is a direct summand of $(T, P)$ and $(T', P')$. In this case we write $T'=\mu_X(T)$ if $X$ is an
indecomposable $\la$-module satisfying either $T = U\oplus X$ or $P = Q\oplus X$.
\end{defn}
The following result shows that support $\tau$-tilting modules `complete' tilting modules from the viewpoint of mutation.
\begin{thm}\label{thm:mutationsupport}\cite[Theorem 2.17]{AIR}
Let $\la$ be a finite dimensional $k$-algebra. Then any basic almost support $\tau$-tilting
pair $(U,Q)$ for $\la$ is a direct summand of exactly two basic support $\tau$-tilting pairs $(T, P)$ and $(T', P')$
for $\la$. Moreover we have $\{\Fac T, \Fac T'\}=\{\Fac U, ^{\perp}(\tau U)\cap Q^{\perp}\}$.
\end{thm}

\subsection{Functorially finite torsion classes}
Let $\la$ be a finite dimensional $k$-algebra and $\tau$ be the Auslander-Reiten translation. We denote by $\mathsf{K}^{{\rm b}}(\proj\la)$ the homotopy category of bounded complexes of finitely generated projective $\la$-modules. We recall the definition of functorially finite torsion classes.
\medskip

We say that a full subcategory $\T$ of $\mod\la$ is a \textit{torsion class} if it is closed under factor modules and extensions and an object $X$ in $\T$ is Ext-\textit{projective} if Ext$_{\la}^1(X, -)|_{\T}=0$. We denote by $P(\T)$ the direct sum of one copy of each of the indecomposable Ext-projective objects in $\T$ up to isomorphism \cite{Ho, Sma}.
\medskip

Let $X$ be a module in $\mod\la$. A morphism $f:T_0\rightarrow X$ is called a \textit{right $\T$-approximation} of $X$ if $T_0\in \T$ and Hom$_{\la}(-, f)|_{\T}$ is surjective. If any module in $\mod\la$ has a right $\T$-approximation, we call $\T$ \textit{contravariantly finite} in $\mod\la$. Dually, a \textit{left $\T$-approximation} and a \textit{covariantly finite subcategory} are defined. We say that $\T$ is \textit{functorially finite} if it is both covariantly finite and contravariantly finite.
\medskip

We denote by f-tors$\la$ the set of functorially finite torsion classes in $\mod\la$. The following result gives a relationship between support $\tau$-tilting $\la$-modules and functorially finite torsion classes in $\mod\la$.
\begin{thm}\label{thm:f-tors}
\cite[Theorem 2.6]{AIR} There is a bijection
$${\rm s}\tau\mbox{-}{\rm tilt}\la \longleftrightarrow {\rm f}\mbox{-}{\rm tors}\la$$
given by ${\rm s}\tau\mbox{-}{\rm tilt}\la\ni M\mapsto \Fac M\in {\rm f}\mbox{-}{\rm tors}\la$ and ${\rm f}\mbox{-}{\rm tors}\la\ni \T\mapsto P(\T)\in {\rm s}\tau\mbox{-}{\rm tilt}\la$, where $\Fac M$ is the subcategory of $\mod\la$ consisting of all objects which are factor modules of finite direct sums of copies of $M$.
\end{thm}

Under this bijection, we can give a partial order on s$\tau$-{\rm tilt}$\la$.
\begin{defn}
For $M, N\in\ ${\rm s}$\tau$-{\rm tilt}$\la$, we write $M\geq N$ if $\ \Fac M\supseteq\Fac N$.
\end{defn}
The relation $\geq$ gives a partial order on s$\tau$-{\rm tilt}$\la$. The following proposition is very important for mutation of support $\tau$-tilting modules.
\begin{prop}\label{prop:criterion}
\cite[Definition-Proposition 2.26]{AIR} Let $T=X\oplus U$ and $T'$ be support $\tau$-tilting $\Lambda$-modules such that $T'=\mu_X(T)$ for some indecomposable $\Lambda$-module $X$. Then either $T>T'$ or $T'>T$. Moreover $T>T'$ (respectively, $T<T'$) if and only if $X\notin \Fac U$ (respectively, $X\in \Fac U$).
\end{prop}

\subsection{Serre functors}
Following Bondal and Kapranov \cite{BK}, we give the definition of Serre functors.
\begin{defn}
Let $\D$ be a $k$-linear triangulated category with finite dimensional Hom-spaces. A Serre functor $\mathbb{S}:\D\rightarrow\D$ is a $k$-linear equivalence with bifunctorial isomorphisms
$${\rm Hom}_{\D}(A, B)\simeq D{\rm Hom}_{\D}(B, \mathbb{S}A)$$
for any $A, B\in \D$, where $D$ is the duality over $k$.
\end{defn}
In \cite{RVdB}, Reiten and Van den Bergh proved that if $\D$ has a Serre functor $\mathbb{S}$, then $\D$ has Auslander-Reiten triangles. Moreover, if $\tau_\D$ is the Auslander-Reiten translation in $\D$, then $\mathbb{S}\simeq \tau_\D[1]$. We say that a triangulated category $\D$ with a Serre functor $\mathbb{S}$ is {\em $n$-Calabi-Yau} if $\mathbb{S}\simeq [n]$.

\subsection{Cluster-tilting objects}\label{subsect:cto}
Let $\D$ be a $k$-linear, Hom-finite triangulated category with a Serre functor $\mathbb{S}$. An important class of objects in $\D$ are the cluster-tilting objects, which have many nice properties. Following Iyama and Yoshino \cite{IY}, we give the definitions of cluster-tilting objects and related objects as follows.

\begin{defn}
\begin{enumerate}[(1)]
\item An object $T$ in $\D$ is called {\rm rigid} if ${\rm Ext}_{\D}^1(T, T)=0$.
\item An object $T$ in $\D$ is called {\rm maximal rigid} if it is rigid and maximal with respect to the property: $\add T=\{X\in\D\ |\ {\rm Ext}_{\D}^1(T\oplus X, T\oplus X)=0\}$.
\item An object $T$ in $\D$ is called {\rm cluster-tilting} if
$$\add T=\{X\in\D\ |\ {\rm Ext}_{\D}^1(T, X)=0\}=\{X\in\D\ |\ {\rm Ext}_{\D}^1(X, T)=0\}.$$
\end{enumerate}
\end{defn}
Throughout this paper, we denote by rigid$\D$ (respectively, c-tilt$\D$) the set of isomorphism classes of basic rigid (respectively, cluster-tilting) objects in $\D$.
For two objects $X$ and $Y$ in $\D$, we define by $X\ast Y$ a full subcategory of $\D$ whose objects are all such $M\in \D$ with triangles $$X_0\longrightarrow M \longrightarrow Y_0 \longrightarrow X_0[1],$$
where $X_0\in \add X$ and $Y_0\in \add Y$.
Let $\tau_\D$ be the Auslander-Reiten translation in $\D$ and denote by $F=\tau_\D^{-1}[1]$. We have the following results, which will be used frequently in this paper.
\begin{thm}\label{FT}
\cite{IY, KZ} Let $T$ be a cluster-tilting object in $\D$. Then we have the following
\begin{enumerate}[(a)]
\item $\D=T\ast T[1]$.
\item $F$ is an auto-equivalence of $\D$ satisfying $FT\cong T$.
\end{enumerate}
\end{thm}
\begin{thm}\label{cto}
\cite{KR, KZ} Let $T$ be a cluster-tilting object in $\D$, and let $\la=\End^{op}_{\D}(T)$. Then the functor
$\overline{(-)}\coloneqq {\rm Hom}_{\D}(T,-):\D\longrightarrow \mod\la$ induces an equivalence
\begin{equation}\label{equi}
\D/[T[1]]\s{\sim}\longrightarrow \mod\la,
\end{equation}
and $\la$ is a Gorenstein algebra of Gorenstein dimension at most 1.
\end{thm}
This theorem tells us that all $\la$-modules have projective dimension zero, one or infinity. In order to characterize the
$\la$-modules of infinite projective dimension, Beaudet, Br\"{u}stle and Todorov \cite{BBT} introduced the following definition.
\begin{defn}
Let $X$ be an object in $\D$. The ideal of $\End^{op}_{\D}(T[1])$ given by all endomorphisms that factor through $X$ is called {\rm factorization ideal} through $X$, denoted by $I_X$.
\end{defn}
It is easy to see that $I_{M\oplus N}= I_M+ I_N$, for any two objects $M$ and $N$ in $\D$. The main theorem in \cite{BBT} is the following.
\begin{thm}\label{thm:projdim}
Let $T$ be a cluster-tilting object in $\D$ and let $\la=\End^{op}_{\D}(T)$. If $M$ is an indecomposable object in $\D$ which does not belong to $\add T[1]$, then the $\la$-module {\rm Hom}$_\D(T,M)$ has infinite projective dimension if and only if the factorization ideal $I_M$ is non-zero.
\end{thm}
We keep the notation of Theorem \ref{cto} and denote by iso$\D$ the set of isomorphism classes of objects in $\D$. By the equivalence ($\ref{equi}$), we have a bijection
\begin{align}\label{bijection}
\widetilde{(-)}:\mbox{iso}\D & \longleftrightarrow\mbox{iso}(\mod\la)\times\mbox{iso}(\proj\la)\\
X=X'\oplus X'' & \longmapsto \quad \widetilde{X}\coloneqq (\overline{X'}, \overline{X''[-1]}),\nonumber
\end{align}
where $X''$ is a maximal direct summand of $X$ which belongs to $\add T[1]$. Under this bijection, a lot of work to study the relationships between objects in $\D$ and modules in $\mod \la$, see for example \cite{AIR, B, CZZ, FL, HJ, Smi, YZZ}.
In particular, we denote by c-tilt$_T\D$ the set of isomorphism classes of basic cluster-tilting objects in $\D$ which do not have non-zero direct summands in $\add T[1]$,
 then we have the following result.
\begin{thm}\cite[Theorem 4.1]{AIR}
If $\D$ is 2-CY, then the bijection $\widetilde{(-)}$ induces bijections
$${\rm rigid}\D\leftrightarrow \tau\text{-}{\rm rigid}\la,\quad {\rm c}\text{-}{\rm tilt}\D\leftrightarrow {\rm s}\tau\text{-}{\rm tilt}\la,\quad and \quad {\rm c}\text{-}{\rm tilt}_T\D\leftrightarrow \tau\text{-}{\rm tilt}\la.$$
\end{thm}
However, this result does not hold if $\D$ is not 2-CY. Here we consider an easy example.

\begin{exm}\label{exm:2.15}
{\upshape Let $Q$ be the quiver $$\setlength{\unitlength}{0.03in}\xymatrix{1 \ar[r]^{\alpha} & 2\ar[r]^{\beta} & 3}.$$
We consider the $3$-cluster category $\D =D^b(kQ)/\tau_Q^{-1}[3]$ of type $A_3$, where $\tau_Q$ is the Auslander-Reiten translation in $D^b(kQ)$ (see \cite{K1, K2, T, Z08} for details).
 Then $\D$ is a 4-Calabi-Yau triangulated category and its AR-quiver is as follows:}
\end{exm}
\vspace{-4ex}
\begin{center}
\scalebox{0.65}{\includegraphics{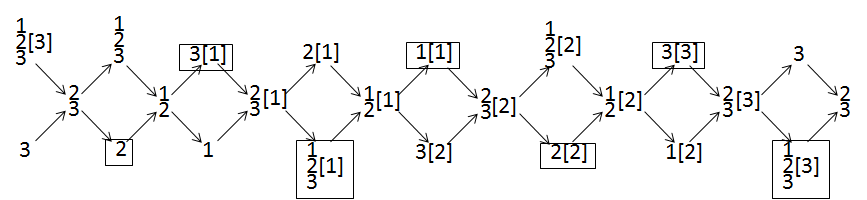}}
\end{center}
{\setlength{\jot}{-3pt}
The direct sum $T=2\oplus 3[1]\oplus \begin{aligned}1\\2\\3\end{aligned}[1]\oplus 1[1]\oplus 2[2]\oplus 3[3]\oplus \begin{aligned}1\\2\\3\end{aligned}[3]$ gives a cluster-tilting object, and the opposite algebra of the endomorphism algebra $\la=\End^{op}_{\D}(T)$ is given by the following quiver with $rad^2=0$.}
\begin{center}
$\begin{xy} /r12mm/:
(6.6,0),{  \xypolygon7{
~*{\xypolynode}
~><{@/_.0ex/}
}}
\end{xy}$
\end{center}
The AR-quiver of $\mod\la$ is as follows:
\vspace{-2ex}
\begin{center}
$$
\xymatrix@!@C=0.01cm@R=0.5cm{  
\txt{1\\2}\ar[dr]&&&&\txt{6\\7}\ar[dr]&&&&*+[o][F-]{\txt{4\\5}}\ar[dr]&&&&\txt{2\\3}\ar[dr]&&\\  
&1\ar[dr]&&7\ar[ur]\ar@{.>}[ll]&&6\ar[dr]\ar@{.>}[ll]&&*+[o][F-]{5}\ar[ur]\ar@{.>}[ll]&&4\ar@{.>}[ll]\ar[dr]&&3\ar@{.>}[ll]\ar[ur]&&2\ar@{.>}[ll]\ar[dr]&\\
&&\txt{7\\1}\ar[ur]&&&&*+[o][F-]{\txt{5\\6}}\ar[ur]&&&&\txt{3\\4}\ar[ur]&&&&\txt{1\\2}  }
$$
\end{center}

{\setlength{\jot}{-3pt}    
It is easy to see that $M=\begin{aligned}4\\5\end{aligned}\oplus 5\oplus\begin{aligned}5\\6\end{aligned}$ is a support $\tau$-tilting $\la$-module, but the object in $\D$ corresponding to $M$ is $1[1]\oplus \begin{aligned}1\\2\end{aligned}[1]\oplus \begin{aligned}1\\2\\3\end{aligned}[1]\oplus 3\oplus \begin{aligned}1\\2\\3\end{aligned}\oplus 1\oplus 2[1]$, which is not a cluster-tilting object since it has self-extensions.}
\medskip

In next section, we shall investigate an important class of objects in $\D$, which correspond to support $\tau$-tilting $\la$-modules bijectively.

\section{Relative cluster tilting objects and support $\tau$-tilting modules}\label{sect:3}
In this section, we study the following objects in triangulated categories.
\begin{defn}
Let $\D$ be a triangulated category with cluster-tilting objects.
\begin{itemize}
\item An object $X$ in $\D$ is called {\rm relative rigid} if there exists a cluster-tilting object $T$ such that $[T[1]](X, X[1])=0$. In this case, $X$ is also called $T[1]$-{\rm rigid}.    
\item An object $X$ in $\D$ is called {\rm relative cluster tilting} (respectively, {\rm almost relative cluster tilting}) if there exists a cluster-tilting object $T$ such that $X$ is $T[1]$-rigid and $|X|=|T|$ (respectively, $|X|=|T|-1$), where $|X|$ denotes the number of non-isomorphic indecomposable direct summands of $X$. In this case, $X$ is also called $T[1]$-{\rm cluster tilting} (respectively, {\rm almost} $T[1]$-{\rm cluster tilting}).
\end{itemize}
\end{defn}
\begin{rem}
Any rigid object in $\D$ is relative rigid.
\end{rem}

The following easy observation shows that relative cluster tilting objects can be regarded as a generalization of cluster-tilting objects.
\begin{prop}
If $\D$ admits a Serre functor, then for any cluster-tilting object $T$, cluster-tilting objects in $\D$ are $T[1]$-cluster tilting.
\end{prop}
\begin{proof}
Let $M$ be a cluster-tilting object in $\D$. Clearly, it is $T[1]$-rigid. By Corollary 2.15 in \cite{YZZ}, all basic cluster-tilting objects have the same number of indecomposable direct summands. Thus, $|M|=|T|$. Hence $M$ is $T[1]$-cluster tilting.
\end{proof}

Throughout this section, we assume that $\D$ is a $k$-linear, Hom-finite triangulated category with cluster-tilting objects and a Serre functor $\mathbb{S}$. Let $T$ be a cluster-tilting object in $\D$ and $\la=\End^{op}_{\D}(T)$ the opposite algebra of the endomorphism algebra of $T$. Our aim in this section is to show that there is a close relationship between $T[1]$-cluster tilting objects in $\D$ and support $\tau$-tilting $\la$-modules.
\medskip

Let $\tau_\D$ be the Auslander-Reiten translation in $\D$. We first give the following proposition, which indicates that $T[1]$-rigid objects and rigid objects coincide in some cases.
\begin{prop}\label{prop:T[1]-rigid}
For two objects $M$ and $N$ in $\D$, $[T[1]](M, N[1])=0$ and $[T[1]](N, \tau_\D M)=0$ if and only if ${\rm Hom}_{\D}(M, N[1])=0$. In particular, if $\D$ is 2-CY, then $X$ is $T[1]$-rigid if and only if $X$ is rigid.
\end{prop}
\begin{proof}
We show the `if' part. If ${\rm Hom}_{\D}(M, N[1])=0$, then $[T[1]](M, N[1])=0$. By Serre duality, we have $${\rm Hom}_{\D}(N, \tau_\D M)\simeq {\rm Hom}_{\D}(N[1], \mathbb{S}M)\simeq D\Hom_\D(M, N[1])=0.$$
Thus we obtain $[T[1]](N, \tau_\D M)=0$.
\medskip

We show the `only if' part. Since $T$ is a cluster-tilting object, by Theorem \ref{FT}, we know there exists a triangle
$$T_0\s{g}\longrightarrow N \s{f}\longrightarrow T_1[1] \s{h}\longrightarrow T_0[1]$$ with $T_0, T_1\in \add T$. Thus we have a commutative diagram of exact sequences
 $$\xymatrix@!@C=0.5cm@R=0.01cm{
(T_1[1],\tau_\D M)\ar[r]^{\cdot f}\ar[d]_{\wr}&  (N,\tau_\D M)\ar[r]\ar[d]^{\wr} & (T_0, \tau_\D M)\ar[r]\ar[d]_{\wr} & (T_1, \tau_\D M)\ar[d]_{\wr}  \\
D(M, T_1[2])\ar[r]^{D(f[1]\cdot)}          &  D(M, N[1])\ar[r]^{D(g[1]\cdot)} &            D(M,T_0[1])\ar[r]^{D(h\cdot)}   & D(M,T_1[1]). }
$$
Since Im$(\cdot f)=\{af\ |\ a\in \Hom_\D(T_1[1],\tau_\D M)\}\subseteq [T[1]](N, \tau_\D M)=0$, we know that
\begin{equation}\label{equi:1}
{\rm Ker}D(g[1]\cdot)={\rm Im}D(f[1]\cdot)\simeq {\rm Im}(\cdot f)=0.
\end{equation}
$$\xymatrix@!@C=0.9cm@R=0.8cm{
N\ar[r]^{f}\ar[d] & T_1[1]\ar[r]^{h}\ar[ld] & T_0[1]\ar[r]^{g[1]} & N[1]  \\
\tau_\D M         &                         & M\ar[u]^{b}\ar@{.>}[lu]^{\exists c}  }
$$
Take any $b\in$ Hom$_\D(M, T_0[1])$. Since $[T[1]](M, N[1])=0$, we have $g[1]b=0$. Thus there exists $c:M\rightarrow T_1[1]$ such that $b=hc$, which implies that \begin{eqnarray*}
(h\cdot):\Hom_\D(M, T_1[1]) & \longrightarrow & \Hom_\D(M, T_0[1]) \\
 c & \longmapsto & hc.
\end{eqnarray*} 
 is surjective. Therefore, $D(h\cdot)$ is injective and $$\text{Im}D(g[1]\cdot)=\text{Ker}D(h\cdot)=0.$$ Combining this with (\ref{equi:1}), we know that $D$Hom$_\D(M, N[1])=0$ and Hom$_\D(M, N[1])$ vanishes.
\end{proof}

The following lemma plays an important role in this paper. This was proved by Palu \cite{P} in case $\D$ is a 2-CY category, and proved in \cite{YZZ} for general case. For the convenience of the readers, we give a simple proof in the following.
\begin{lem}\label{Palu}
Let $\D$ be a triangulated category with a Serre functor $\mathbb{S}$ and a cluster-tilting object $T$. Then for any objects $X$ and $Y$ in $\D$, there is a  bifunctorial isomorphism
$${\rm Hom}_{\D/[T]}(\tau_\D^{-1}Y, X)\simeq D[T](X[-1], Y).$$
\end{lem}
\begin{proof}
Since $T$ is a cluster-tilting object, by Theorem \ref{FT}, we know there exists a triangle
$$T_1 \longrightarrow T_0 \longrightarrow X \s{\xi}\longrightarrow T_1[1]$$
 in $\D$ with $T_{0}$ and $T_{1}$ in $\add T$. Applying $\Hom_\D(-, Y)$ to it, we have a map
\begin{eqnarray*}
\varphi :  \Hom_{\D}(T_{1}, Y) & \longrightarrow & \Hom_{\D}(X[-1],Y) \\
 f & \longmapsto & f\xi[-1].
\end{eqnarray*}
It is easy to see that $\mbox{Im}\varphi\simeq [T](X[-1], Y)$. Since the category $\D$ has a Serre functor $\mathbb{S}$, we have isomorphisms
$$D\Hom_{\D}(T_{1}, Y)\simeq \Hom_{\D}(\mathbb{S}^{-1}Y, T_{1})\simeq \Hom_{\D}(\tau_\D^{-1}Y, T_{1}[1]),$$
$$D\Hom_{\D}(X[-1], Y)\simeq \Hom_{\D}(\mathbb{S}^{-1}Y, X[-1])\simeq \Hom_{\D}(\tau_\D^{-1}Y, X).$$

$$\xymatrix@!@C=1cm@R=0.01cm{
D\Hom_{\D}(X[-1], Y)\ar[r]^{D\varphi}\ar[d]^{\simeq}  &  D\Hom_{\D}(T_{1}, Y)\ar[d]_{\simeq} \\
\Hom_{\D}(\tau_\D^{-1}Y, X)\ar[r]^{\exists\phi}        &   \Hom_{\D}(\tau_\D^{-1}Y, T_{1}[1])  }
$$

Thus, $D\varphi$ is isomorphic to a map
\begin{eqnarray*}
\phi:  \Hom_{\D}(\tau_\D^{-1}Y, X) & \longrightarrow & \Hom_{\D}(\tau_\D^{-1}Y, T_{1}[1]) \\
 g & \longmapsto & \xi g.
\end{eqnarray*}
Note that $\mbox{Ker}{\phi}=[T](\tau_\D^{-1}Y, X)$.
Hence, we have isomorphisms
\begin{eqnarray}
D[T](X[-1], Y)\simeq D\mbox{Im}\varphi \simeq\mbox{Im}D\varphi &\simeq &\mbox{Im}{\phi}   \nonumber \\
  & \simeq &  \Hom_{\D}(\tau_\D^{-1}Y, X)/\mbox{Ker}{\phi}   \nonumber \\
  & \simeq  & \Hom_{\D}/[T](\tau_\D^{-1}Y, X).   \nonumber
\end{eqnarray}
\end{proof}
We keep the notation of bijection (\ref{bijection}). Let $X$ be an object in $\D$, we define $$|\wte{X}|= |(\overline{X'}, \overline{X''[-1]})|\coloneqq |\overline{X'}|+|\overline{X''[-1]}|,$$
it is easy to see that $|\wte{X}|=|X|$.
\medskip

From now on, we denote by $T[1]$-rigid$\D$ (respectively, $T[1]$-tilt$\D$) the set of isomorphism classes of basic $T[1]$-rigid (respectively, $T[1]$-cluster tilting) objects in $\D$, by $T[1]$-tilt$_T\D$ the set of objects in $T[1]$-tilt$\D$ which do not have non-zero direct summands in $\add T[1]$, and by
$$T[1]\mbox{-}{\rm tilt}_T^0\D:=\{\ X\in T[1]\mbox{-}{\rm tilt}_T\D\ |\ I_X=0\}.$$
 On the other hand, we denote by tilt$\la$ the set of isomorphism classes of basic tilting modules in $\mod\la$. The following correspondences are our main result in this section.

\begin{thm}\label{thm:mainthm}
Let $\D$ be a $k$-linear, Hom-finite triangulated category with a Serre functor $\mathbb{S}$ and a cluster-tilting object $T$, and let $\la={\rm End}^{op}_{\D}(T)$. Then the bijection $\widetilde{(-)}$ in (\ref{bijection}) induces the following bijections
\begin{eqnarray*}
T[1]\mbox{-}{\rm rigid}\D & \s{(a)}\longleftrightarrow & \tau\mbox{-}{\rm rigid}\la,\\
T[1]\mbox{-}{\rm tilt}\D & \s{(b)}\longleftrightarrow & {\rm s}\tau\mbox{-}{\rm tilt}\la,\\
T[1]\mbox{-}{\rm tilt}_T\D & \s{(c)}\longleftrightarrow & \tau\mbox{-}{\rm tilt}\la,\\
\mbox{and }T[1]\mbox{-}{\rm tilt}_T^0\D & \s{(d)}\longleftrightarrow & {\rm tilt}\la.
\end{eqnarray*}
\end{thm}
\begin{proof}
 By Corollary 6.5(3)(iii) in \cite{IY}, if $M\in\D$ has no nonzero indecomposable direct summands of $T[1]$, then
 $$\overline{\tau_\D M}\simeq \tau\overline{M}.$$ Combining this with Lemma \ref{Palu}, for any object $X$ in $\D$, we have
$$\Hom_\la(\overline{X'}, \tau\overline{X'})=\Hom_\la(\overline{X'}, \overline{\tau_\D X'})\simeq {\rm Hom}_{\D/[T[1]]}(X, \tau_\D X')\simeq D[T[1]](\tau_\D X'[-1], \tau_\D X).$$
Further, by Theorem \ref{FT}, we have
\begin{equation}\label{equi:2}
\Hom_\la(\overline{X'}, \tau\overline{X'})\simeq D[FT[1]](F\tau_\D X'[-1], F\tau_\D X)\simeq D[T[1]](X', X[1]).
\end{equation}
In the similar way, we know $$\Hom_\la(\overline{X''[-1]}, \overline{X'})\simeq D[T[1]](X[-1], \tau_\D X''[-1])=D[T[1]](X[-1], F^{-1}X'').$$
Note that $F^{-1}X''\in \add T[1]$. Then
\begin{eqnarray}\label{equi:3}
\Hom_\la(\overline{X''[-1]}, \overline{X'}) & \simeq & D\Hom_\D(X[-1], F^{-1}X'') \nonumber \\
 & \simeq & \Hom_\D(F^{-1}X'', \mathbb{S}X[-1]) \qquad \ \text{(Serre duality)} \nonumber \\
 & \simeq & [T[1]](F^{-1}X'', \tau_\D X)        \nonumber \\
 & \simeq & [T[1]](X'', X[1]). \qquad \qquad\quad \text{(Theorem \ref{FT})}
\end{eqnarray}

(a) By equalities (\ref{equi:2}) and (\ref{equi:3}), we know that $X$ is a $T[1]$-rigid object in $\D$ if and only if $$[T[1]](X', X[1])=[T[1]](X'', X[1])=0$$ if and only if $(\overline{X'}, \overline{X''[-1]})$ is a $\tau$-rigid pair for $\la$.
\medskip

(b) Note that $|\wte{X}|=|(\overline{X'}, \overline{X''[-1]})|=|X|$ and $|\la|=|T|$. We know that (a) induces a bijection between $T[1]$-cluster tilting objects in $\D$ and support $\tau$-tilting pairs for $\la$. By Proposition \ref{uniqueness}(b), we have proved the assertion.
\medskip

(c) This assertion is clear.
\medskip

(d) By Theorem \ref{thm:projdim}, we only need to show that $\tau$-tilting modules whose projective dimension are at most one are precisely tilting modules.
This is immediate from the fact that if the projective dimension of a $\la$-module $M$ is at most one, then $M$ is $\tau$-rigid if and only if it is rigid, i.e. Ext$_\la^1(M, M)=0$ (using AR duality, see \cite{ASS}).
\end{proof}

We define
 $${\rm c}\mbox{-}{\rm tilt}_T^0\D:=\{\ X\in {\rm c}\mbox{-}{\rm tilt}_T\D\ |\ I_X=0\}$$
and end this section with the following direct consequences.
\begin{cor}\label{cor:maincor1}
Let $\D$ be a $k$-linear, Hom-finite triangulated category with a Serre functor and a cluster-tilting object $T$, and let $\la={\rm End}^{op}_{\D}(T)$. Then we have the following.
\begin{enumerate}[(1)]
\item For any $T[1]$-rigid object $X$ in $\D, |X|\leq |T|$. In particular, for any maximal rigid object $X$ in $\D, |X|\leq |T|$.
\item Any $T[1]$-rigid object in $\D$ is a direct summand of some $T[1]$-cluster tilting object in $\D$.
\item Any basic almost $T[1]$-cluster tilting object in $\D$ is a direct summand of exactly two basic $T[1]$-cluster tilting objects in $\D$.
\item If $\D$ is 2-CY, then we have bijections \begin{eqnarray*}
{\rm rigid}\D &  \longleftrightarrow & \tau\mbox{-}{\rm rigid}\la,\\
{\rm c}\mbox{-}{\rm tilt}\D &  \longleftrightarrow & {\rm s}\tau\mbox{-}{\rm tilt}\la,\\
{\rm c}\mbox{-}{\rm tilt}_T\D & \longleftrightarrow & \tau\mbox{-}{\rm tilt}\la,\\
\mbox{and }{\rm c}\mbox{-}{\rm tilt}_T^0\D &  \longleftrightarrow & {\rm tilt}\la.
\end{eqnarray*}
\end{enumerate}
\end{cor}
\begin{proof}
(1) This is immediate from the bijection (a) in Theorem \ref{thm:mainthm}.
\medskip

(2) Let $X$ be a $T[1]$-rigid object in $\D$, then $\wte{X}= (\overline{X'}, \overline{X''[-1]})$ is a $\tau$-rigid pair for $\la$. We may assume that $\overline{X''[-1]}=\la e$, where $e$ is an idempotent of $\la$. By Proposition \ref{uniqueness}, we know $\overline{X'}$ is a $\tau$-rigid $(\la/\langle e\rangle)$-module. Using Theorem \ref{thm:tauBongartz}, we know there exists a $\tau$-tilting $(\la/\langle e\rangle)$-module $S$ such that $\overline{X'}$ is a direct summand of $S$. Thus, $(S, \la e)$ is a support $\tau$-tilting pair for $\la$ and $(\overline{X'}, \la e)$ is a direct summand of $(S, \la e)$. Hence, it follows from Theorem \ref{thm:mainthm} and Proposition \ref{uniqueness} that $X$ is a direct summand of some $T[1]$-cluster tilting object in $\D$.
\medskip

(3) This assertion follows from Theorem \ref{thm:mutationsupport} and the bijection (b) in Theorem \ref{thm:mainthm} immediately.
\medskip

(4) This assertion follows from Proposition \ref{prop:T[1]-rigid} and the fact that an object $X$ in $\D$ is cluster-tilting if and only if it is rigid and $|X|=|T|$ (see \cite{ZZ}).
\end{proof}

\begin{rem}
\begin{itemize}
\item In a triangulated category $\D$ with cluster-tilting objects and a Serre functor $\mathbb{S}$, any cluster-tilting object is maximal rigid. Conversely, any maximal rigid object $M$ satisfying $\mathbb{S}(M)\simeq M[2]$ is cluster-tilting (see \cite{AIR,ZZ,YZZ}). But we do not know whether any maximal rigid object in $\D$ is cluster-tilting.
\item The first 3 bijections in Corollary \ref{cor:maincor1}(4) are known by \cite{AIR}.
\end{itemize}
\end{rem}

\section{Mutation of $T[1]$-cluster tilting objects and an application}\label{sect:4}
As previous, we assume that $\D$ is a $k$-linear, Hom-finite triangulated category with a cluster-tilting object $T$ and a Serre functor $\mathbb{S}$, and $\la=\End^{op}_{\D}(T)$ is the opposite algebra of the endomorphism algebra of $T$.
 In this section, we first introduce a natural partial order on the set of $T[1]$-cluster tilting objects, then prove our main result on complements for almost $T[1]$-cluster tilting objects in $\D$. As an application, we give a partial answer to a question posed in \cite{AIR}.

\subsection{A partial order}
For an object $M$ in $\D$, we denote by $M\ast[T[1]]$ a full subcategory of $\D$ whose objects are all such $X\in \D$ with triangles $$M_X\longrightarrow X \s{\eta_X}\longrightarrow C_X \longrightarrow M_X[1],$$
where $M_X\in \add M$ and the morphism $\eta_X$ factors through $\add T[1]$.
\begin{defn}
For $M, N\in T[1]${\rm -tilt}$\D$, we write $M\geq N$ if $M\ast[T[1]]\supseteq N\ast[T[1]]$.
\end{defn}
The main result in this subsection is the following.
\begin{thm}
The relation $\geq$ gives a partial order on $T[1]${\rm -tilt}$\D$.
\end{thm}
\begin{proof}
We only need to show that for $M, N\in T[1]${\rm -tilt}$\D$, $M\geq N$ and $N\geq M$ imply $M=N$. We assume $M\ast [T[1]]=N\ast[T[1]]$. Since $N\in N\ast [T[1]]=M\ast[T[1]]$, there exists a triangle
$M_N\s{a}\longrightarrow N \s{\eta_N}\longrightarrow C_N \longrightarrow M_N[1],$ where $M_N\in \add M$ and $\eta_N$ factors through $\add T[1]$.
$$
\xymatrix@!@C=0.6cm@R=0.6cm{
M_N\ar[r]^a&N\ar[r]^{\eta_N}&C_N\ar[r]&M_N[1]\\
&T_0\ar[u]^{f_2}\ar@{.>}[ul]^b  \\
&M[-1]\ar[u]^{f_1}  }
$$
$\forall f\in [T](M[-1], N),$ there are two morphisms $f_1:M[-1]\rightarrow T_0$ and $f_2:T_0\rightarrow N$ such that $f=f_2f_1$, where $T_0\in \add T$. Thus, $\eta_Nf_2=0$ and there exists $b:T_0\rightarrow M_N$ such that $f_2=ab$. Since $M$ is $T[1]$-cluster tilting, we have $f=f_2f_1=a(bf_1)=0$. This implies that $[T[1]](M, N[1])=0$. Dually, $[T[1]](N, M[1])=0$. Hence $M\oplus N$ is $T[1]$-rigid. By Corollary \ref{cor:maincor1}(1), we have
$$|T|=|M|=|N|\leq |M\oplus N|\leq |T|,$$ which forces that $|M\oplus N|=|M|=|N|$. Therefore, $M=N$.
\end{proof}
The following observation is crucial in this subsection.
\begin{lem}\label{lem:suffnece}
For any two objects $M$ and $X$ in $\D$, $X\in M\ast[T[1]]$ if and only if $\overline{X'}\in \Fac\overline{M'}$ in $\mod\la$.
\end{lem}
\begin{proof}
(1) If $X\in M\ast[T[1]]$, then there exists a triangle
\begin{equation}\label{equi:4}
M_X \s{g}\longrightarrow X \s{\eta_X}\longrightarrow C_X \longrightarrow M_X[1],
\end{equation}
where $M_X\in \add M$ and the morphism $\eta_X$ factors through $\add T[1]$. Applying $\overline{(\ )}$ to (\ref{equi:4}), we have an exact sequence
\begin{equation*}
\overline{M_X} \s{\overline{g}}\longrightarrow \overline{X'} \s{\overline{\eta_X}}\longrightarrow \overline{C_X} \longrightarrow \overline{M_X[1]},
\end{equation*}
Since $\eta_X$ factors through $\add T[1]$, we have Im$\overline{\eta_X}=0$. Thus $\Im\overline{g}=\text{Ker}\overline{\eta_X}=\overline{X'}$, i.e. $\overline{g}$ is surjective. Because $\overline{M_X}\in \add\overline{M'}$, we obtain $\overline{X'}\in \Fac\overline{M'}$.
\medskip

(2) Conversely, let $\overline{X'}\in \Fac\overline{M'}$ in $\mod\la$, then there is a surjection
$\overline{M'}^n \s{h} \longrightarrow \overline{X'}\longrightarrow 0$. By Theorem \ref{cto}, we have an equivalence \begin{equation}\label{equi}
\D/[T[1]]\s{\sim}\longrightarrow \mod\la,
\end{equation}
thus there exists a morphism $f:(M')^n\rightarrow X'$ in $\D$ such that $\overline{f}=h$. Take a triangle $(M')^n \s{f} \rightarrow X' \s{g}\rightarrow C \rightarrow (M')^n[1]$. Applying $\overline{(\ )}$ to it, we get an exact sequence $$\overline{M'}^n \s{h} \longrightarrow \overline{X'} \s{\overline{g}}\longrightarrow \overline{C} \longrightarrow \overline{(M')^n[1]}.$$
Because $h$ is surjective, we have $\Im\overline{g}=0$. Since $T$ is a cluster-tilting object in $\D$, by Theorem \ref{FT}, we have the following triangles
$$
\xymatrix@!@C=0.4cm@R=0.4cm{
&        T_0\ar[d]^a\\
(M')^n\ar[r]^f  &  X'\ar[r]^{g}\ar[d]^b  & C\ar[r]  &  (M')^n[1],  \\
&T_1[1]\ar[d]\ar@{.>}[ur]^{\exists c}  \\
&T_0[1]}
$$ where $T_0, T_1\in \add T.$ Since the image of $\overline{g}:\Hom_\D(T, X')\longrightarrow\Hom_\D(T, C)$ is zero, we get $ga=0$. Thus $g$ factors through $T_1[1]$. Because $(M')^n\in \add M$, we know $X'\in M\ast[T[1]]$, so is $X$.
\end{proof}
This lemma gives the following direct consequence.
\begin{prop}\label{prop:twoposets}
For any two objects $M$ and $N$ in $\D$, $M\ast[T[1]]\subseteq N\ast[T[1]]$ in $\D$ if and only if $\Fac\overline{M'}\subseteq\Fac\overline{N'}$ in $\mod\la$. In particular, the bijection (b) in Theorem \ref{thm:mainthm} induces an isomorphism of two partially ordered sets.
\end{prop}
We introduce the following notion in triangulated categories, which is an analog of Ext-projective modules in module categories.
\begin{defn}
Let $\D$ be a triangulated category and $\T$ be a subcategory of $\D$. We say that an object $X\in \T$ is {\rm relative cluster projective} if there exists a cluster-tilting object $T$ such that $[T[1]](X, \T[1])=0$. In this case, $X$ is also called $T[1]$-{\rm projective}.
\end{defn}
We first give the following interesting observation.
\begin{lem}\label{lem:sumT[1]-proj}
Let $T$ be a cluster-tilting object and $M$ be a basic $T[1]$-cluster tilting object in $\D$. Any indecomposable $T[1]$-projective object in $M\ast[T[1]]$ is a direct summand of $M$.
\end{lem}
\begin{proof}
Let $X$ be an indecomposable $T[1]$-projective object in $M\ast[T[1]]$. Then
\begin{equation}\label{equi:t[1]-proj1}
\begin{split}
[T[1]](X, M[1]) &=0\\
\mbox{and }[T[1]](X, X[1])&=0.
\end{split}
\end{equation}
Since $X\in M\ast[T[1]]$, we have a triangle $M_X\s{g}\longrightarrow X \s{\eta_X}\longrightarrow C_X \longrightarrow M_X[1],$ where $M_X\in \add M$ and $\eta_X$ factors through $\add T[1]$. $\forall f\in [T](M[-1], X),$ there are two morphisms $f_1:M[-1]\rightarrow T_0$ and $f_2:T_0\rightarrow X$ such that $f=f_2f_1$, where $T_0\in \add T$.
\begin{center}
\begin{tabular}{cc}
$\xymatrix@!@C=0.6cm@R=0.6cm{
&M[-1]\ar[d]^{f_1} \\
&T_0\ar[d]^{f_2}\ar@{.>}[dl]_a  \\
M_X\ar[r]^g&X\ar[r]^{\eta_X}&C_X\ar[r]&M_X[1]
  }$

& \quad
$\xymatrix@!@C=0.6cm@R=0.6cm{
&X[-1]\ar[d]^{h_1} \\
&T_0\ar[d]^{h_2}\ar@{.>}[dl]_b  \\
M_X\ar[r]^g&X\ar[r]^{\eta_X}&C_X\ar[r]&M_X[1]
  }$  \\
\text{Figure 1}  &   \text{Figure 2}
\end{tabular}
\end{center}
Since $\eta_X$ factors through $\add T[1]$, we have $\eta_Xf_2=0$. Thus, there exists $a:T_0\longrightarrow M_X$ such that $f_2=ga$.
Because $M$ is $T[1]$-rigid, we have $f=f_2f_1=g(af_1)=0.$ Therefore
\begin{equation}\label{equi:t[1]-proj2}
[T[1]](M, X[1])=0.
\end{equation}

It follows from the equalities (\ref{equi:t[1]-proj1}) and (\ref{equi:t[1]-proj2}) that $[T[1]](M\oplus X, (M\oplus X)[1])=0$, i.e. $M\oplus X$ is $T[1]$-rigid. By Corollary \ref{cor:maincor1}(1), we know that $X\in \add M$.
\end{proof}
Thanks to this lemma, we can consider the direct sum of one copy of each of the indecomposable $T[1]$-projective objects in $M\ast[T[1]]$ up to isomorphism and denote it by $P_{T[1]}(M\ast[T[1]])$.
\medskip

We put
$$T[1]\mbox{-}{\rm tilt}\D\ast T[1]:=\{M\ast[T[1]]\ |\ M\in T[1]\mbox{-tilt}\D \}.$$ Then we have the following result.
\begin{prop}\label{prop:projective}
(1) For any object $M$ in $T[1]${\rm -tilt}$\D$, $P_{T[1]}(M\ast[T[1]])=M$.

(2) there is a one-to-one correspondence
\begin{equation}\label{bij:1}
T[1]\mbox{-}{\rm tilt}\D\ast T[1]\longrightarrow T[1]\mbox{-}{\rm tilt}\D\end{equation} given by $T[1]\mbox{-}{\rm tilt}\D\ast T[1]\ni \T\mapsto P_{T[1]}(\T)\in T[1]\mbox{-}{\rm tilt}\D.$ The inverse is given by $M\mapsto M\ast[T[1]].$
\end{prop}

\begin{proof}
(1) Let $X$ be an indecomposable $T[1]$-projective object in $M\ast[T[1]]$. By Lemma \ref{lem:sumT[1]-proj}, we know $X\in \add M$. On the other hand, we can use similar arguments as in the proof of Lemma \ref{lem:sumT[1]-proj} to show that each direct summand of $M$ is a $T[1]$-projective object in $M\ast[T[1]]$.
\medskip

(2) The assertion follows from (1) immediately.
\end{proof}
With the notation of the above discussion, we give the following result.
\begin{thm}The bijection in Proposition \ref{prop:projective} is compatible with bijection in Theorem \ref{thm:f-tors}. In other words,
we have a commutative diagram
$$\xymatrix@!@C=1.6cm@R=0.1cm{
T[1]\mbox{-}{\rm tilt}\D\ast T[1]\ar[r]^{P_{T[1]}(-)}_{\rm 1:1}\ar[d]_{{\rm Hom}_\D(T, -)}^{\rm 1:1}  &  T[1]\mbox{-}{\rm tilt}\D\ar[d]^{{\rm Hom}_\D(T, -)}_{\rm 1:1} \\
{\rm f}\mbox{-}{\rm tors}\la\ar[r]^{P(-)}_{\rm 1:1}        &   {\rm s}\tau\mbox{-}{\rm tilt}\la  }
$$
in which each map is a bijection. The upper horizontal map is given in Proposition \ref{prop:projective}, the lower horizontal map is given in Theorem \ref{thm:f-tors} and the right vertical map is given in Theorem \ref{thm:mainthm}.
\end{thm}
\begin{proof}We consider the left vertical map. For any $M\ast[T[1]]$ in $T[1]\mbox{-}{\rm tilt}\D\ast T[1]$, where $M\in T[1]\mbox{-}{\rm tilt}\D$, by Lemma \ref{lem:suffnece}, we get
$${\rm Hom}_\D(T, M\ast[T[1]])=\Fac\overline{M'}\in {\rm f}\mbox{-}{\rm tors}\la.$$
By Theorem \ref{thm:f-tors}, we know that the left vertical map is bijective.
It is easy to see that this diagram commutes.
\end{proof}

\subsection{Mutation of $T[1]$-cluster tilting objects}
Let $\D$ be a $k$-linear, Hom-finite triangulated category with a cluster-tilting object $T$ and a Serre functor $\mathbb{S}$, and let $\la=\End^{op}_{\D}(T)$ be the opposite algebra of the endomorphism algebra of $T$.
\medskip

Recall that for an almost $T[1]$-cluster tilting object $U$ in $\D$, by Corollary \ref{cor:maincor1} we know that there are two non-isomorphic $T[1]$-cluster tilting objects $M=U\oplus X$ and $N=U\oplus Y$ in $\D$. Under bijections in Theorem \ref{thm:mainthm}, we get two support $\tau$-tilting pairs $\widetilde{M}=(\overline{M'}, \overline{M''[-1]})$ and $\widetilde{N}$ for $\la$ which have $\widetilde{U}$ as a direct summand. By Proposition \ref{prop:criterion}, we know that either $\overline{M'} > \overline{N'}$ or $\overline{M'} < \overline{N'}$. Using Proposition \ref{prop:twoposets}, we know that either $M>N$ or $M<N$. Thus, we can introduce the following notion.

\begin{defn}
For an almost $T[1]$-cluster tilting object $U$ in $\D$, by Corollary \ref{cor:maincor1} and Proposition \ref{prop:twoposets}, we know that there are two $T[1]$-cluster tilting objects $M=U\oplus X$ and $N=U\oplus Y$ in $\D$ satisfying $M>N$, where $X$ and $Y$ are indecomposable. In this case, we call $(M, N)$ an $U$-{\rm mutation pair} and $X$ and $Y$ two {\rm complements} to $U$. In this section, we also say that $N$ is a {\rm left mutation} of $M$ and $M$ is a {\rm right mutation} of $N$ and we write $N=\mu_X^L(M)$ and $M=\mu_Y^R(N)$.
\end{defn}
Given an almost $T[1]$-cluster tilting object in $\D$, the main result in this subsection shows that starting with a complement, we can calculate the other one by an exchange triangle, which is constructed from a left approximation or a right approximation.
\begin{thm}\label{thm:mutation}
Let $M=U\oplus X$ be a basic $T[1]$-cluster tilting object in $\D$, where $X$ is indecomposable. Then we have the following.
\begin{enumerate}[(1)]
\item If $X\in U\ast[T[1]]$, we take a triangle
\begin{equation}\label{equi:5}
Y \s{g}\longrightarrow U_1 \s{f}\longrightarrow X \s{h}\longrightarrow Y[1], \tag{$\star$}
\end{equation}
where $f$ is a minimal right $(\add U)$-approximation. In this case, $Y$ is another complement to $U$ and $U\oplus Y>M$.
\item If $X\notin U\ast[T[1]]$, we take a triangle
\begin{equation}\label{equi:6}
X \s{g}\longrightarrow U_2 \s{f}\longrightarrow Y \s{h}\longrightarrow X[1], \tag{$\star\star$}
\end{equation}
where $g$ is a minimal left $(\add U)$-approximation. In this case, $Y$ is another complement to $U$ and $U\oplus Y<M$.
\end{enumerate}
\end{thm}

To prove this theorem, we need some preparations. First we need the following easy observation.
\begin{lem}\label{lem:facthr}
The map $h: X\rightarrow Y[1]$ in (\ref{equi:5}) factors through $\add T[1]$.
\end{lem}
\begin{proof}
Since $X\in U\ast[T[1]]$, we have a triangle $U_X \s{c_X}\longrightarrow X \s{\eta_X}\longrightarrow C_X \longrightarrow U_X[1],$
 where $U_X\in \add U$ and the morphism $\eta_X$ factors through $\add T[1]$. Because the map $f$ in (\ref{equi:5}) is a right $(\add U)$-approximation, there exists $i: U_X\rightarrow U_1$ such that $c_X=fi$. By the octahedral axiom, we have a commutative diagram
$$\xymatrix@!@C=0.6cm@R=0.6cm{
               &U_X\ar@{.>}[d]_{\exists i}\ar@{=}[r]  & U_X\ar[d]^{c_X} \\
Y\ar@{=}[d]\ar[r]^g & U_1\ar[d]\ar[r]^f    & X\ar[r]^h\ar[d]^{\eta_X} & Y[1]\ar@{=}[d]  \\
Y\ar[r]        & X' \ar[d]\ar[r]     & C_X\ar[d]\ar[r]          & Y[1]  \\
               & U_X[1]\ar@{=}[r]      & U_X[1]
  }
$$
 of triangles. Thus $h$ factors through $\eta_X$, which implies that $h$ factors through $\add T[1]$.
\end{proof}
The following lemma plays an important role in the proof of Theorem \ref{thm:mutation}.
\begin{lem}\label{lem:minapproximation}
The map $g: Y\rightarrow U_1$ in (\ref{equi:5}) is a minimal left $(\add M)$-approximation. In particular, $g$ is a minimal left $(\add U)$-approximation.
\end{lem}
\begin{proof}
Take any map $a:Y\rightarrow M_0$, where $M_0\in \add M$. By Lemma \ref{lem:facthr}, we may assume that there are two morphisms $h_1:X[-1]\rightarrow T_0$ and $h_2:T_0\rightarrow Y$ such that $h[-1]=h_2h_1$, where $T_0\in \add T$.
$$\xymatrix@!@C=0.6cm@R=0.6cm{
X[-1]\ar[r]^{h[-1]}\ar[d]_{h_1}  &  Y\ar[rd]_a\ar[r]^g & U_1\ar[r]^f\ar@{.>}[d]    & X\ar[r]^h & Y[1]   \\
T_0\ar[ur]_{h_2}                 &&  M_0     }
$$
Noticing that $$ah[-1]=(ah_2)h_1\in [T](X[-1], M_0)$$ and $M=U\oplus X$ is $T[1]$-rigid, we get $ah[-1]=0$. This implies that $a$ factors through $g$, and hence $g$ is a left $(\add M)$-approximation.
\medskip

Now we show that $g$ is a left minimal map. If this were not true, then there would be a decomposition $U_1=U_{11}\oplus U_{12}$ such that $$g=\binom{g_0}{0}: Y\longrightarrow U_{11}\oplus U_{12}.$$
Consider the following triangles
$$Y \s{g_0}\longrightarrow U_{11} \longrightarrow Y' \s{h}\longrightarrow Y[1],$$
$$0       \longrightarrow U_{12}  \longrightarrow U_{12} \longrightarrow 0.$$
Thus we get a triangle
\begin{equation}\label{equi:6}
Y \s{g=\binom{g_0}{0}}\longrightarrow U_{11}\oplus U_{12} \longrightarrow Y'\oplus U_{12} \longrightarrow Y[1],
\end{equation}
Comparing the triangle (\ref{equi:5}), we obtain that $X\simeq Y'\oplus U_{12}$. Since $X$ is indecomposable, we would have $X\simeq U_{12}\in \add U$. This is a contradiction and our claim follows.
\end{proof}
The following results are also crucial.
\begin{lem}\label{lem:complement}
The object $Y$ in (\ref{equi:5}) is indecomposable and it is not in $\add M$.
\end{lem}
\begin{proof}
Suppose that there is a decomposition $Y=Y_1\oplus Y_2$ with $Y_1$ and $Y_2$ nonzero. Since $\add U$ is functorially finite in $\D$, we can get two triangles
$$Y_1 \s{g_1}\longrightarrow U_{13}\s{f_1} \longrightarrow X_{1} \longrightarrow Y_1[1] \mbox{ and } Y_2 \s{g_2}\longrightarrow U_{14}\s{f_2} \longrightarrow X_{2} \longrightarrow Y_2[1],$$
where $g_1$ and $g_2$ are two minimal left $(\add U)$-approximation. Thus by Lemma \ref{lem:minapproximation} the direct sum of these two triangles is $Y \s{g}\longrightarrow U_1 \s{f}\longrightarrow X \s{h}\longrightarrow Y[1]$, which implies that $X=X_1\oplus X_2$. Since $X$ is indecomposable, we may assume that $X_1=0$ and $X_2=X$. Thus $U_{13}\simeq Y_1$ and $$f=(0, f_2): U_{13}\oplus U_{14}\longrightarrow X.$$ This is a contradiction because $f$ is a right minimal map. Therefore $Y$ is indecomposable.
\medskip

Now we show that $Y$ is not in $\add M$. If this were not true, we would have $U_1\simeq Y$ and $X=0$ by Lemma \ref{lem:minapproximation}. This is a contradiction and our claim follows.
\end{proof}

\begin{lem}\label{lem:bigandsmall}
Assume that $X$ and $Z$ are two non-isomorphic complements to an almost $T[1]$-cluster tilting object $U$ in $\D$. Then $U\oplus X>U\oplus Z$ if and only if $X\notin U\ast [T[1]]$.
\end{lem}
\begin{proof}
By Proposition \ref{prop:twoposets}, we know that $U\oplus X>U\oplus Z$ if and only if $\overline{U\oplus X}>\overline{U\oplus Z}$. This is equivalent to $\overline{X'}\notin \Fac\overline{U'}$ by Proposition \ref{prop:criterion}. Thanks to Lemma \ref{lem:suffnece}, this is equivalent to $X\notin U\ast [T[1]]$.
\end{proof}

Now we are ready to prove Theorem \ref{thm:mutation}.
\medskip

(1) We first show that $Y$ is another complement to $U$. Using Lemma \ref{lem:complement}, we only need to show that $U\oplus Y$ is $T[1]$-rigid.
Take any map $a\in [T[1]](U, Y[1])$. Since $a$ factors through $\add T[1]$ and $U$ is $T[1]$-rigid, we have $g[1]a=0$.
$$\xymatrix@!@C=0.6cm@R=0.6cm{
 Y\ar[r]^g & U_1\ar[r]^f    & X\ar[r]^h & Y[1]\ar[r]^{g[1]}  &  U_1[1]   \\
   &&&U\ar[u]_{a}\ar@{.>}[ul]_{\exists a_1}\ar@{.>}[ull]^{\exists a_2}              }
$$
Thus there exists $a_1: U\rightarrow X$ such that $a=ha_1$. Observing that $f$ is a right $(\add U)$-approximation, we know there exists $a_2: U\rightarrow U_1$ such that $a_1=fa_2$. Thus $a=ha_1=(hf)a_2=0$ and hence $$[T[1]](U, Y[1])=0.$$ For any morphism $b\in [T[1]](Y, U[1])$, we know that there are two morphisms $b_1: Y\rightarrow T_1[1]$ and $b_2:T_1[1]\rightarrow U[1]$ such that $b=b_2b_1$, where $T_1\in \add T$.
\begin{center}
\begin{tabular}{cc}
$\xymatrix@!@C=0.6cm@R=0.6cm{
X[-1]\ar[r]^{h[-1]} & Y\ar[r]^g\ar[d]^{b_1} & U_1\ar[r]^f\ar@{.>}[ld]^{\exists b_3}    & X  \\
   & T_1[1]\ar[d]^{b_2}\\
   & U[1]             }
$
& \quad
$\xymatrix@!@C=0.6cm@R=0.6cm{
X[-1]\ar[r]^{h[-1]} & Y\ar[r]^g\ar[d]^{c_1} & U_1\ar[r]^f\ar@{.>}[ld]^{\exists c_3}    & X  \\
   & T_2[1]\ar[d]^{c_2}\\
   & Y[1]             }$
   \\
\text{Figure 3}  &   \text{Figure 4}
\end{tabular}
\end{center}
By Lemma \ref{lem:facthr}, we know $h[-1]$ factors through $\add T$, which implies that $b_1h[-1]=0$. Thus there exists $b_3: U_1\rightarrow T_1[1]$ such that $b_1=b_3g$. Because $U$ is $T[1]$-rigid, we have $b=b_2b_1=(b_2b_3)g=0$. Hence $$[T[1]](Y, U[1])=0.$$
 It remains to show that $Y$ is $T[1]$-rigid. In the similar way (see Figure 4), we know that, for any map $c\in [T[1]](Y, Y[1]), c=(c_2c_3)g$. Since $c_2c_3\in T[1]](U, Y[1])=0$, we have $c=0$ and $[T[1]](Y, Y[1])=0$. Therefore, $Y$ is another complement to $U$.
 \medskip

 In order to show that $U\oplus Y>M$, by Lemma \ref{lem:bigandsmall}, we need to show that $Y\notin U\ast[T[1]]$.  If this were not true, by Lemma \ref{lem:suffnece}, we would have $\overline{Y'}\in\Fac\overline{U'}$. Since $X\in U\ast[T[1]]$, similarly we have $\overline{X'}\in\Fac\overline{U'}$. This contradicts with Proposition \ref{prop:criterion} and our claim follows.
\medskip

(2) Let $Y$ be another complement to $U$. Since $X\notin U\ast[T[1]]$, we have that $U\oplus Y<M$ and $Y\in U\ast[T[1]]$ By Lemma \ref{lem:bigandsmall}. Using (1) and Lemma \ref{lem:minapproximation},
\newenvironment{prf}{\noindent {\bf  }}{\hfill $\Box$}
\begin{prf}we know the assertion follows immediately.
\end{prf}
\medskip

\subsection{An application}
We end this section with an application of mutation. We first recall the following definition.
\begin{defn}\cite{AIR}
Let $\la$ be a finite dimensional $k$-algebra and $U$ be a basic $\tau$-rigid $\la$-module. We call $P(^{\bot}(\tau U))$ the {\rm Bongartz completion} of $U$.
\end{defn}
Note that $U\in P(^{\bot}(\tau U))$ and thus any $\tau$-rigid $\la$-module is a direct summand of some $\tau$-tilting $\la$-module. The following observation is needed in this subsection.
\begin{prop}\label{prop:Boncriterion}
\cite[Definition-Proposition 2.26]{AIR} Let $\la$ be a finite dimensional $k$-algebra and let $T= X\oplus U$ and $T'$ be support $\tau$-tilting $\Lambda$-modules such that $T'=\mu_X(T)$ for some indecomposable $\Lambda$-module $X$. If $T$ is a Bongartz completion of $U$, then $T>T'$ and $X\notin \Fac U$.
\end{prop}

 In \cite{AIR}, the authors gave the following result to calculate left mutation of support $\tau$-tilting modules by exchange sequences.
\begin{thm}\label{thm:Bongartz}
Assume that $\la$ is a finite dimensional $k$-algebra. Let $M_{\la}=X_\la\oplus U_\la$ be a basic $\tau$-tilting $\la$-module which is the Bongartz completion of $U_\la$, where $X_\la$ is indecomposable. Let
\begin{equation}\label{equi:7}
X_\la \s{g_\la} \longrightarrow U'_\la \s{f_\la} \longrightarrow Y_\la \longrightarrow 0   \tag{$\star\star\star$}
\end{equation}
be an exact sequence, where $g_\la$ is a minimal left $(\add U_\la)$-approximation. Then we have the following.
\begin{enumerate}[(a)]
\item If $U_\la$ is not sincere, then $Y_\la=0$. In this case $U_\la=\mu_{X_\la}^L(M_\la)$ holds and this is a basic support $\tau$-tilting $\la$-module which is not $\tau$-tilting.
\item If $U_\la$ is sincere, then $Y_\la$ is a direct sum of copies of an indecomposable $\la$-module $Y_1$ and is not in $\add M_\la$. In this case $Y_1\oplus U_\la=\mu_{X_\la}^L(M_\la)$ holds and this is a basic $\tau$-tilting $\la$-module.
\end{enumerate}
\end{thm}
Furthermore, the authors posed the following question.
\begin{que}\label{que:air}
 Is $Y{_\la}$ always indecomposable in Theorem \ref{thm:Bongartz}(b)?
\end{que}
In this subsection,  we give a positive answer to this question when $\la$ is an endomorphism algebra of a cluster-tilting object. More precisely, there is a cluster-tilting object $T$ in a triangulated category $\D$ with a Serre functor $\mathbb{S}$ such that $\la=\End^{op}_{\D}(T)$.
\medskip

Since $M_{\la}=X_\la\oplus U_\la$ is a $\tau$-tilting $\la$-module, we know there exists a $T[1]$-cluster tilting object $M=X\oplus U$ in $\D$ such that $\overline{X}=X_\la$ and $\overline{U}=U_\la$ by Theorem \ref{thm:mainthm}.
Note that $M_{\la}=X_\la\oplus U_\la$ is the Bongartz completion of $U_\la$. By Proposition \ref{prop:Boncriterion}, we know $X_\la\notin \Fac U_\la$. By Lemma \ref{lem:suffnece}, we have
$X\notin U\ast[T[1]]$. Thus we can use the triangle (\ref{equi:6}) to obtain another complement $Y$ to $U$. Our main result of this subsection is the following.

\begin{thm}\label{thm:triangletosequence}
The exact sequence (\ref{equi:7}) in $\mod\la$ is induced from the triangle (\ref{equi:6}) in $\D$.
\end{thm}
\begin{proof}
Applying $\overline{(\ )}$ to (\ref{equi:6}) and using Lemma \ref{lem:facthr}, we have an exact sequence
\begin{equation}\label{equi:8}
\overline{X} \s{\overline{g}}\longrightarrow \overline{U_2} \s{\overline{f}}\longrightarrow \overline{Y} \longrightarrow 0.
\end{equation}
Since $g$ is a left $(\add U)$-approximation, we know that $\overline{g}$ is a left $(\add U_\la)$-approximation. Now we show that $\overline{g}$ is a left minimal map. If this were not true, then there would be a decomposition $\overline{U_2}=W_{1}\oplus W_{2}$ and an exact sequence
 $$\overline{X} \s{\overline{g}=\binom{g_0}{0}}\longrightarrow W_1\oplus W_2 \s{\overline{f}}\longrightarrow \overline{Y} \longrightarrow 0$$ in $\mod\la$. Thus $W_{2}$ would be a direct summand of $\overline{Y}$. Since $Y\notin \add U$, we know this is a contradiction and our claim follows. Hence the exact sequences (\ref{equi:8}) and (\ref{equi:7}) are coincident.
\end{proof}
The following consequence is direct, which gives a partial anwser to Question \ref{que:air}.
\begin{cor}\label{cor:ques}
Let $\D, \la$ be as above. Then $Y_\la$ is always indecomposable in Theorem \ref{thm:Bongartz}(b).
\end{cor}

\section{Examples}\label{sect:5}
\bigskip
\begin{exm}\label{exm:1}
{\upshape Let $A=kQ/I$ be a self-injective algebra given by the quiver $$Q:
\setlength{\unitlength}{0.03in}\xymatrix{1 \ar@<0.5ex>[r]^{\alpha}_{\ } & 2\ar@<0.5ex>[l]^{\beta}}$$ and $I=<\alpha\beta\alpha\beta, \beta\alpha\beta\alpha>$. Let $\D$ be the stable module category
${\rm \underline{\mod}}A$ of $A$. This is a triangulated category with a Serre functor (it is not 2-CY). We describe the AR-quiver of $\mod A$ in the following picture:}
\end{exm}
\vspace{-2ex}
\begin{center}
$
\xymatrix@!@C=0.01cm@R=0.2cm{
&\txt{1\\2\\1\\2}\ar[dr]&&\txt{2\\1\\2\\1}\ar[dr]&\\
*+[o][F-]{\txt{2\\1\\2}}\ar[ur]\ar[dr]\ar@{--}[u]&&\txt{1\\2\\1}\ar@{.>}[ll]\ar[ur]\ar[dr]&&\txt{2\\1\\2}\ar@{.>}[ll]\ar@{--}[u]\\
&\txt{2\\1}\ar[dr]\ar[ur]&&\txt{1\\2}\ar[dr]\ar[ur]\ar@{.>}[ll]&\\
 1\ar[ur]\ar@{--}[uu]&&*+[o][F-]{2}\ar@{.>}[ll]\ar[ur]&&1\ar@{.>}[ll]\ar@{--}[uu] }
$

\medskip

Figure 5
\end{center}
where the leftmost and rightmost columns are identified. Thus, we also get the AR-quiver of ${\rm \underline{\mod}}A$ by deleting the first row in
 Figure 5.{\setlength{\jot}{-3pt}
The direct sum $T=2\oplus \begin{aligned}2\\1\\2\end{aligned}$ is a cluster-tilting object in $\D$.} The opposite algebra of endomorphism algebra $\la=\End^{op}_{\D}(T)=kQ'/I'$ is given by the quiver $Q'$:
$\xymatrix{{\rm a} \ar@<0.5ex>[r]^{\gamma}_{\ } & {\rm b}\ar@<0.5ex>[l]^{\delta}}$ and $I'= <\gamma\delta, \delta\gamma>$. The AR-quiver of $\mod\la$ is
\begin{center}
$$
\xymatrix@!@C=0.01cm@R=0.5cm{
\txt{a\\b}\ar[dr]&&&&\txt{a\\b}\\
&\txt{a}\ar[dr]&&\txt{b}\ar[ur]\ar@{.>}[ll]\\
&&\txt{b\\a}\ar[ur]  }
$$
\end{center}
We depict $T[1]$-cluster tilting objects in $\D$ and support $\tau$-tilting modules in $\mod \la$ as follows (the encircled objects are cluster-tilting objects)
\begin{center}
\begin{tabular}{cc}
$
\xymatrix@!@C=0.6cm@R=0.6cm{
&*+[o][F-]{\txt{2\\1\\2}\oplus 2}\ar[dr]\ar[dl]& \\
\txt{1\\2}\oplus 2\ar[d]&& \txt{2\\1\\2}\oplus \txt{2\\1}\ar[d]  \\
\txt{1\\2}\oplus 1\ar[dr]&&\txt{2\\1}\oplus \txt{1\\2\\1}\ar[dl]\\
& *+[o][F-]{1\oplus \txt{1\\2\\1}} }
$
& \qquad\quad
$
\xymatrix@!@C=0.6cm@R=0.6cm{
&\txt{a\\b}\oplus\txt{b\\a}\ar[dr]\ar[dl]& \\
\txt{b\\a}\oplus \txt{b}\ar[d]&& \txt{a\\b}\oplus \txt{a}\ar[d]  \\
\txt{b} \ar[dr]&&\txt{a}\ar[dl] \\
& \txt{0} }
$

   \\
$T[1]$-\text{tilt}$\D$  &   \text{\qquad\qquad s}$\tau$-\text{tilt}$\la$
\end{tabular}
\end{center}
\medskip

\begin{exm}\label{exm:2}
{\upshape Let $Q$ be the quiver $\setlength{\unitlength}{0.03in}\xymatrix{1 \ar[r]^{\alpha} & 2}.$
Assume that $\tau_Q$ is the Auslander-Reiten translation in $D^b(kQ)$. We consider the repetitive cluster category $\D=D^b(kQ)/\langle \tau_Q^{-2}[2]\rangle$ introduced by Zhu in \cite{Z11}, whose objects are the same in $D^b(kQ)$, and whose morphisms are given by $$\Hom_{D^b(kQ)/\langle \tau_Q^{-2}[2]\rangle}(X, Y)=\bigoplus_{i\in \mathbb{Z}}\Hom_{D^b(kQ)}(X, (\tau_Q^{-2}[2])^iY).$$  }
\end{exm}
 It is shown in \cite{Z11} that $\D$ is a triangulated category with a Serre functor $\mathbb{S}$. Note that it is not 2-CY (but it is a fractional Calabi-Yau category with CY-dimension $\frac{4}{2}$). The AR-quiver of $\D$ is as follows:
\begin{center}\label{f:1}$
\xymatrix@!@C=0.5cm@R=1cm{  
&\txt{1\\2}\ar@{.}[l]\ar[dr]&&*+[o][F-]{2[1]}\ar@{.>}[ll]\ar[dr]&&1[1]\ar@{.>}[ll]\ar[dr]&&*+[o][F-]{\txt{1\\2}[2]}\ar@{.>}[ll]\ar[dr]&&2[3]\ar@{.>}[ll]\ar[dr]&&\txt{1\\2}\ar@{.>}[ll]  \\
2\ar[ur] &&*+[o][F-]{1}\ar@{.>}[ll] \ar[ur]&&\txt{1\\2}[1]\ar@{.>}[ll]\ar[ur] && 2[2]\ar@{.>}[ll] \ar[ur]&&*+[o][F-]{1[2]}\ar@{.>}[ll] \ar[ur]&&2 \ar@{.>}[ll]\ar[ur]  }$

Figure 6
\end{center}
{\setlength{\jot}{-3pt}
 The direct sum $T=1\oplus 2[1]\oplus \begin{aligned}1\\2\end{aligned}[2]\oplus 1[2]$} of the encircled indecomposable objects gives a cluster-tilting object. Note that the opposite algebra of the endomorphism algebra $\la=\End^{op}_{\D}(T)$ is not connected, it is given by the following disconnected quiver: $$\xymatrix{{a} \ar[r] & {b}} \qquad \xymatrix{{c} \ar[r] & {d}}$$ with no relations.
The AR-quiver of $\mod\la$ is
\begin{center}\label{f:3}
$
\xymatrix@!@C=0.5cm@R=1cm{  
&\txt{a\\b}\ar[dr]&&&&{d} \ar[dr]&&{c}\ar@{.>}[ll]\\  
{b}\ar[ur]&&{a}\ar@{.>}[ll] &&&&\txt{c\\d} \ar[ur] }$
\end{center}
We use the following picture to describe $T[1]$-cluster tilting objects in $\D$ and support $\tau$-tilting modules in $\mod \la$.
\vspace{-2ex}
\begin{center}
\scalebox{0.13}{\includegraphics{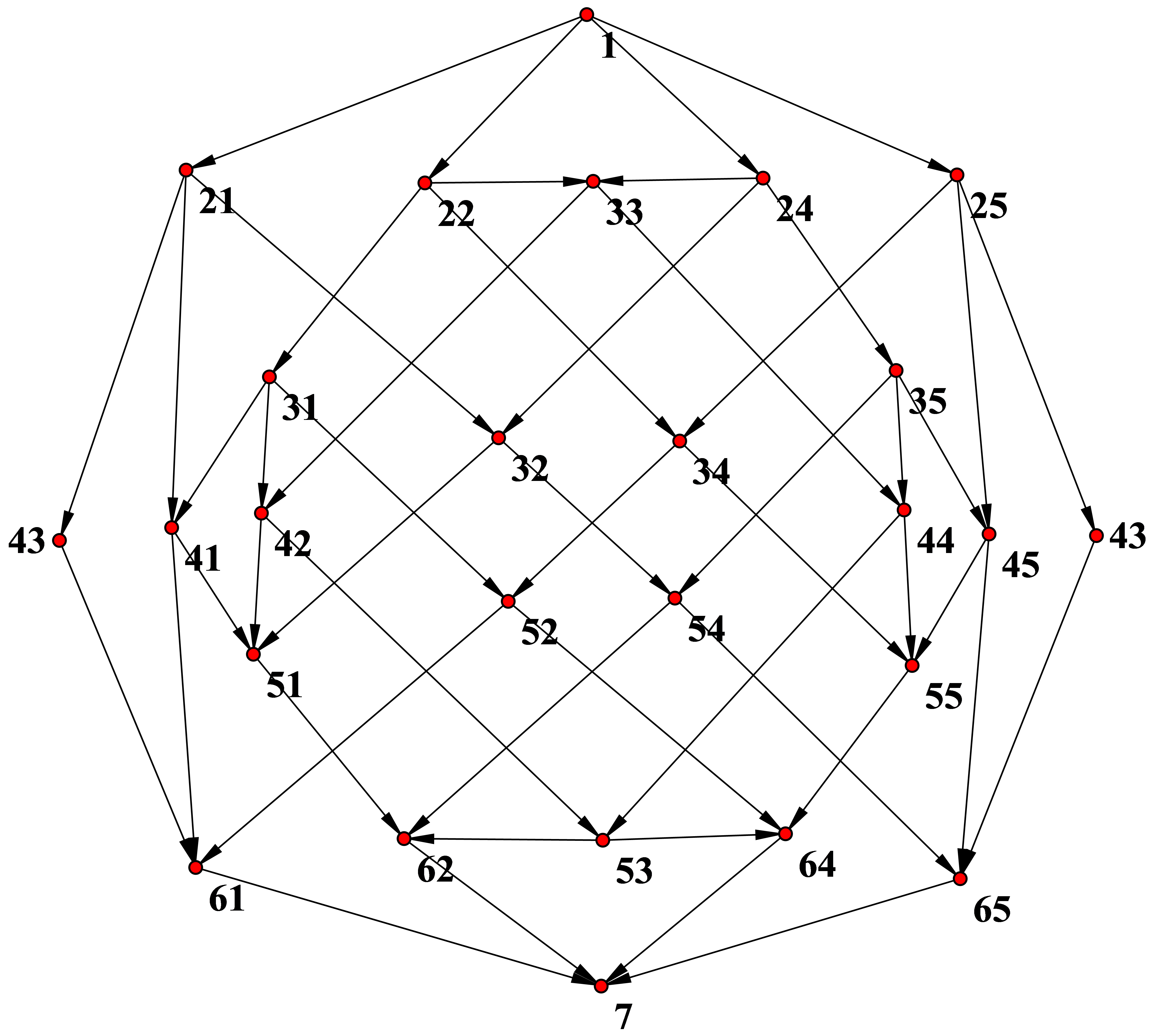}},
\end{center}
where the leftmost and rightmost points are identified, and the cluster-tilting objects are marked by $\clubsuit$.
{\setlength{\jot}{-3pt}

\begin{center}
\begin{tabular}{|c|c|c|}
\hline
Points & $T[1]$-cluster tilting objects & support $\tau$-tilting modules \\\hline
1$\clubsuit$      & $\begin{aligned}1\\2\end{aligned}[2]\oplus 1[2]\oplus 1\oplus 2[1]$ & $\la$\\\hline
21      & $\begin{aligned}1\\2\end{aligned}[2]\oplus 1[2]\oplus 1\oplus 2[2]$ & $\begin{aligned}c\\d\end{aligned}\oplus d\oplus b$\\\hline
22      & $\begin{aligned}1\\2\end{aligned}[2]\oplus 1[2]\oplus \begin{aligned}1\\2\end{aligned}[1]\oplus 2[1]$ & $\begin{aligned}c\\d\end{aligned}\oplus d\oplus a\oplus \begin{aligned}a\\b\end{aligned}$\\\hline
24      & $2[3]\oplus 1[2]\oplus 1\oplus 2[1]$ & $\begin{aligned}c\\d\end{aligned}\oplus c\oplus \begin{aligned}a\\b\end{aligned}\oplus b$\\\hline
25      & $\begin{aligned}1\\2\end{aligned}[2]\oplus \begin{aligned}1\\2\end{aligned}\oplus 1\oplus 2[1]$ & $\begin{aligned}a\\b\end{aligned}\oplus b \oplus d$\\\hline
31      & $\begin{aligned}1\\2\end{aligned}[1]\oplus 1[2]\oplus 1[1]\oplus \begin{aligned}1\\2\end{aligned}[2]$ & $\begin{aligned}c\\d\end{aligned}\oplus d\oplus a$\\\hline
32      & $2[2]\oplus 1[2]\oplus 1\oplus 2[3]$ & $\begin{aligned}c\\d\end{aligned}\oplus b\oplus c$\\\hline
33$\clubsuit$      & $\begin{aligned}1\\2\end{aligned}[1]\oplus 1[2]\oplus 2[3]\oplus 2[1]$ & $\begin{aligned}c\\d\end{aligned}\oplus c \oplus \begin{aligned}a\\b\end{aligned}\oplus a$\\\hline
34      & $\begin{aligned}1\\2\end{aligned}[1]\oplus \begin{aligned}1\\2\end{aligned}[2]\oplus \begin{aligned}1\\2\end{aligned}\oplus 2[1]$ & $\begin{aligned}a\\b\end{aligned}\oplus a\oplus d$\\\hline
35      & $1\oplus 2[1]\oplus 2\oplus 2[3]$ & $\begin{aligned}a\\b\end{aligned}\oplus b\oplus c$\\\hline
41      & $\begin{aligned}1\\2\end{aligned}[2]\oplus 1[2]\oplus 1[1]\oplus 2[2]$ & $\begin{aligned}c\\d\end{aligned}\oplus d$\\\hline
42      & $\begin{aligned}1\\2\end{aligned}[1]\oplus 1[2]\oplus 1[1]\oplus 2[3]$ & $\begin{aligned}c\\d\end{aligned}\oplus a\oplus c$\\\hline

43$\clubsuit$      & $\begin{aligned}1\\2\end{aligned}[2]\oplus 2[2]\oplus 1\oplus \begin{aligned}1\\2\end{aligned}$ & $b\oplus d$\\\hline

44      & $\begin{aligned}1\\2\end{aligned}[1]\oplus 2[3]\oplus 2\oplus 2[1]$ & $\begin{aligned}a\\b\end{aligned}\oplus a\oplus c$\\\hline

45      & $\begin{aligned}1\\2\end{aligned}\oplus 2[1]\oplus 1\oplus 2$ & $\begin{aligned}a\\b\end{aligned}\oplus b$\\\hline

51      & $2[3]\oplus 1[2]\oplus 1[1]\oplus 2[2]$ & $\begin{aligned}c\\d\end{aligned}\oplus c$\\\hline

52      & $\begin{aligned}1\\2\end{aligned}\oplus \begin{aligned}1\\2\end{aligned}[1]\oplus \begin{aligned}1\\2\end{aligned}[2]\oplus 1[1]$ & $a\oplus d$\\\hline

53$\clubsuit$      & $\begin{aligned}1\\2\end{aligned}[1]\oplus 1[1]\oplus 2\oplus 2[3]$ & $a\oplus c$\\\hline

54      & $2\oplus 2[2]\oplus 1\oplus 2[3]$ & $b\oplus c$\\\hline

55      & $\begin{aligned}1\\2\end{aligned}[1]\oplus 2[1]\oplus 2\oplus \begin{aligned}1\\2\end{aligned}$ & $\begin{aligned}a\\b\end{aligned}\oplus a$\\\hline

61      & $\begin{aligned}1\\2\end{aligned}[2]\oplus 1[1]\oplus \begin{aligned}1\\2\end{aligned}\oplus 2[2]$ & $d$\\\hline

62      & $1[1]\oplus 2[2]\oplus 2[3]\oplus 2$ & $c$\\\hline

64      & $\begin{aligned}1\\2\end{aligned}[1]\oplus 1[1]\oplus 2\oplus \begin{aligned}1\\2\end{aligned}$ & $a$\\\hline

65      & $\begin{aligned}1\\2\end{aligned}\oplus 1\oplus 2\oplus 2[2]$ & $b$\\\hline

7$\clubsuit$      & $\begin{aligned}1\\2\end{aligned}\oplus 1[1]\oplus 2\oplus 2[2]$ & $0$\\\hline
\end{tabular}
\end{center}}
\bigskip

{\bf Acknowledgments}

\medskip

The first author would like to thank Wen Chang, Dong Yang, Jie Zhang and Yu Zhou for helpful discussions on the subject. Both authors are grateful to the referee for the valuable comments
and suggestions.

\bigskip

School of Mathematics, Northwest University, Xi'’an 710127, Shaanxi, P. R. China

\textit{E-mail address: }yangwz@nwu.edu.cn

\medskip

Department of Mathematical Sciences, Tsinghua University, 100084, Beijing, P. R. China

\textit{E-mail address: }bzhu@math.tsinghua.edu.cn

\end{document}